\documentclass{amsart}
\usepackage{amsmath,amssymb,amsthm}
\newtheorem{thr}{Theorem}[section]
\newtheorem{lem}[thr]{Lemma}
\newtheorem{cor}[thr]{Corollary}

\newtheorem{claim}[thr]{Claim}
\newtheorem{con}[thr]{Conjecture}

\theoremstyle{definition}

\theoremstyle{remark}
\newtheorem{remr}[thr]{Remark}

\numberwithin{equation}{section}

\def\Mat{\operatorname{Mat}}

\begin{document}

\title[A matrix ring with commuting graph of maximal diameter]{A matrix ring with commuting \\ graph of maximal diameter}

\author{Yaroslav Shitov}
\address{National Research University Higher School of Economics, 20 Myasnitskaya Ulitsa, Moscow 101000, Russia}
\email{yaroslav-shitov@yandex.ru}

\subjclass[2000]{05C50, 12F05}
\keywords{Commuting graph, field extension, matrix theory}

\begin{abstract}
The commuting graph of a semigroup is the set of non-central elements;
the edges are defined as pairs $(u,v)$ satisfying $uv=vu$. We provide
an example of a field $\mathcal{F}$ and an integer $n$ such that the
commuting graph of $\operatorname{Mat}_n(\mathcal{F})$ has maximal
possible diameter, equal to six.
\end{abstract}

\maketitle

\section{Introduction}

The study of commuting graphs of different algebraic structures has attracted considerable attention in recent publications. There are a number of interesting results connecting this area of research with functional analysis~\cite{ABKM}, ring theory~\cite{AGHM}, semigroup theory~\cite{AMK}, and some other branches of mathematics. A significant number of publications are devoted to commuting graphs arising from group theory~\cite{AAM} and linear algebra~\cite{AMRP}. One of the interesting recent results was obtained in~\cite{GiPa}, where the following group-theoretical problem has been solved: What is the maximal possible diameter of the commuting graph of a finite group? Our paper deals with the linear algebraic version of this question. Let $\mathbb{F}$ be a field and $\Mat_n(\mathbb{F})$ be the algebra of $n\times n$ matrices over $\mathbb{F}$. We denote the commuting graph of $\Mat_n(\mathbb{F})$ by $\Gamma(\mathbb{F},n)$. That is, the vertices of $\Gamma(\mathbb{F},n)$ are non-scalar matrices, and the edges are defined as pairs $(U,V)$ satisfying $UV=VU$. S. Akbari, A. Mohammadian, H. Radjavi, and P. Raja proved in 2006 that the distances in commuting graphs of matrix algebras cannot exceed six, and they proposed the following conjecture.

\begin{con}\cite[Conjecture~5]{AMRP}\label{ConjConn5}
If $\Gamma(\mathbb{F},n)$ is a connected graph, then its diameter does not exceed five.
\end{con}

This conjecture is known to be true in the following cases: $n=4$ or $n$ is prime~\cite{Dolz2016, ABM}; $\mathbb{F}$ is algebraically closed~\cite{AMRP} or real closed~\cite{Miguel1}; $\mathbb{F}$ is finite and $n$ is not a square of a prime~\cite{Dolz2016}. The set of matrix pairs realizing the maximal distance is well understood in the case of algebraically closed field~\cite{DKO}. It was proved in~\cite{GutGut} that $\Gamma(\mathbb{F}_2,9)$ is connected and has diameter at least five, showing that Conjecture~\ref{ConjConn5} cannot be strengthened. Conjecture~\ref{ConjConn5} has also been extensively investigated in the case when $\mathbb{F}$ is a more general structure than a field. This conjecture is known to be true when $\mathbb{F}$ is equal to $\mathbb{Z}/m$ with non-prime $m$ or taken from a sufficiently general class of semirings which includes the \textit{tropical semiring}~\cite{DBO, DoOb, non-prime}.

Conjecture~\ref{ConjConn5} has remained open until now, and our paper aims to provide a counterexample for this conjecture. In Section~2, we construct a field over which Conjecture~\ref{ConjConn5} fails. In Section~3, we construct a pair of matrices which realize the maximal distance, equal to six, in the corresponding commuting graph. In Section~4, we conclude the paper by proving a technical claim used in our argument.






\section{The field}

Let us recall some notation. An integral domain $R$ is called a \textit{gcd domain} if any two non-zero elements have a greatest common divisor. An element $v$ in some extension of $R$ is called \textit{integral over $R$} if $v$ is a root of a monic polynomial with coefficients in $R$. A ring $R$ is called \textit{local} if the set $J=J(R)$ of all non-units is an ideal of $R$. A local ring $R$ is \textit{Henselian} if it satisfies the condition of \textit{Hensel's lemma}. That is, for every monic polynomial $f\in R[t]$ whose image $\overline{f}$ in $R[t]/J[t]$ factors into a product $\overline{g_1}\overline{g_2}$ in which $\overline{g_1}$ and $\overline{g_2}$ are both monic and relatively prime, there exist monic polynomials $g_1,g_2\in R[t]$ which are relatively prime and satisfy $f=g_1g_2$ and $\overline{g_i}=g_i+J[t]$.

Let us fix a prime number $p\geqslant 73$. We work with the power series ring $\mathcal{R}=\mathbb{F}_3[[x_{11},x_{12},\ldots,x_{3p\,3p},y,z]]$, which is a gcd domain~\cite[Corollary 3.3]{Buch} and also a local Henselian ring~\cite{San}. The ideal $J(\mathcal{R})$ consists of series with zero constant term, so we have a natural isomorphism between $\mathbb{F}_3$ and $\mathcal{R}/J(\mathcal{R})$. We denote by $\mathcal{Q}$ the field of fractions $\operatorname{Quot}\mathcal{R}$ and by $\mathcal{H}$ the algebraic closure of $\mathcal{Q}$. (We do not provide an explicit construction of $\mathcal{H}$, but we refer the reader to~\cite{Poon} for related questions.) We consider a polynomial $\widetilde{\varphi}=t^p+a_{p-1}t^{p-1}+\ldots+a_1t+a_0$ which is irreducible over $\mathbb{F}_3$, and we define
$$\varphi=t^p+a_{p-1}t^{p-1}+\ldots+a_2t^2+(a_1+y)t+(a_0+z)\in\mathcal{R}[t],$$
where $t$ is an indeterminate, and coefficients lie in $\mathcal{R}$.
The polynomial  $\overline{\varphi}$ equals $\widetilde{\varphi}$ up to the above mentioned isomorphism between $\mathbb{F}_3$ and $\mathcal{R}/J(\mathcal{R})$, so $\overline{\varphi}$ is irreducible over $\mathcal{R}/J(\mathcal{R})$. Therefore, $\varphi$ is irreducible over $\mathcal{R}$ and, since Gauss' lemma holds in gcd domains~\cite[IV.4, Theorem 4.6]{MRR}, $\varphi$ is irreducible over $\mathcal{Q}$ as well.



\begin{lem}\label{LemExten}
Let $\Phi$ be the set of all subfields $F\subset\mathcal{H}$ such that (1) $F$ is a separable extension of $\mathcal{Q}$ and (2) $\varphi$ is irreducible over $F$. Then $\Phi$ has a maximal field $\mathcal{F}$. 
\end{lem}

\begin{proof}
Let $\mathcal{F}'=\{\mathcal{F}_i\}$ be a totally ordered subset of $\Phi$. 
Every element of $\mathcal{F}_0=\cup\mathcal{F}'$ is separable over $\mathcal{Q}$, so $\mathcal{F}_0$ is a separable extension. Assuming $\varphi=\varphi_1\varphi_2$ over $\mathcal{F}_0$, we enumerate the coefficients involved in $\varphi_1$ and $\varphi_2$ as $c_1,\ldots,c_k$, and every $c_j$ lies in some field $F_j\in\mathcal{F}'$. Since $\varphi$ is irreducible over $F_1\cup\ldots\cup F_k\in\mathcal{F}'$, we see that either $\varphi_1$ or $\varphi_2$ is constant, which means that $\varphi$ is irreducible over $\mathcal{F}_0$. Application of Zorn's lemma completes the proof.
\end{proof}

In what follows, $\mathcal{F}$ denotes the field constructed in the previous lemma. 

\begin{lem}\label{LemIntElem2}
Any non-trivial separable extension of $\mathcal{F}$ contains a root of $\varphi$.
\end{lem}

\begin{proof}
Let $F$ be such an extension. Since the field $\mathcal{H}$ is algebraically closed, we can assume without loss of generality that $F\subset \mathcal{H}$. By Lemma~\ref{LemExten}, there are non-constant monic polynomials $\varphi_1,\varphi_2\in F[t]$ such that $\varphi=\varphi_1\varphi_2$. We enumerate by $c_1,\ldots,c_k$ the coefficients appearing in $\varphi_1$ and $\varphi_2$; note that $c_1,\ldots,c_k$ are integral over $\mathcal{R}$ because they are sums of products of roots of $\varphi$. 

The ring $R'=\mathcal{R}[c_1,\ldots,c_k]$ 
is local~\cite[II.7, Ex. 7.7]{FuchsSalce} and Henselian~\cite[III.4, Ex. 4c]{Bour}. Since $\varphi$ is reducible over $R'$, the polynomial $\overline{\varphi}$ is reducible over $R'/J(R')$. Since finite extensions of finite fields are cyclic, the splitting field of $\overline{\varphi}$ has degree $p$ over $\mathcal{R}/J(\mathcal{R})\cong\mathbb{F}_3$. So we see that $\overline{\varphi}$ splits (and, in particular, has a root $u$) over $R'/J(R')$. By Hensel's lemma, there exists $\rho\in R'$ such that $\overline{\rho}=u$ and $\varphi(\rho)=0$.
\end{proof}

Recall that, for every algebraic field extension $K\supset L$, the set of all elements $l\in K$ that are separable over $L$ forms the field $K_{sep}$, which is the unique separable extension of $L$ over which $K$ is purely inseparable~\cite{Lang}. The degree of the extension $K\supset K_{sep}$ is called the \textit{inseparable degree} of $K\supset L$ and is denoted by $[K:L]_i$. The following theorem is an essential step in the proof of our main result.

\begin{thr}\label{LemIntExt}\label{CorIntExt}\label{ThrCon}
The degree of any finite extension $E\supset\mathcal{F}$ is either a power of $3$ or a multiple of $p$. If the degree of $E$ is a multiple of $p$, then $E$ contains a subfield $E'$ which has degree $p$ over $\mathcal{F}$.
\end{thr}

\begin{proof}
The degree of any purely inseparable extension is a power of the characteristic, so we can assume that $E_{sep}\neq\mathcal{F}$. By Lemma~\ref{LemIntElem2}, the polynomial $\varphi$ has a root $\xi\in E_{sep}$, and then $E'=\mathcal{F}(\xi)$ has degree $p$ over $\mathcal{F}$.
\end{proof}



By Theorem~6 of~\cite{ABM}, the second assertion of Theorem~\ref{CorIntExt} is a sufficient condition for the graph $\Gamma(\mathcal{F},3p)$ to be connected. Let us recall that the equality $[K_1:K_3]_i=[K_1:K_2]_i[K_2:K_3]_i$ holds for any tower $K_1\supset K_2\supset K_3$ of algebraic extensions~\cite{Lang}. We denote by $\psi(\tau)$ the polynomial $\varphi(\tau^3)$ and by $\theta$ a root of $\psi$.

\begin{remr}\label{LemNew}
Since every root of $\psi$ has multiplicity three, the only monic separable polynomial of degree at least $p$ that divides $\psi$ is $\sqrt[3]{\psi}$. Note that the elements $\sqrt[3]{y}$ and $\sqrt[3]{z}$ belong to any field containing the coefficients of $\sqrt[3]{\psi}$.
\end{remr}

\begin{lem}\label{LemIrreduc}
The polynomial $\psi$ is irreducible over $\mathcal{F}$.
\end{lem}

\begin{proof}
Assume that $\psi$ is divisible by a polynomial $\psi_1\in\mathcal{F}[\tau]$ of degree $d\notin\{0,3p\}$. If $\psi_1=\varphi_1(\tau^3)$ for some $\varphi_1\in\mathcal{F}[t]$, then $\varphi$ is divisible by $\varphi_1$, a contradiction. In the rest of our proof, we assume without loss of generality that $\psi_1$ is irreducible; if $\psi_1$ is inseparable, then it satisfies the assumption of the previous sentence. If $\psi_1$ is separable and $d<p$, then the polynomial $\psi_2=\psi_1^3$ in turn satisfies the assumption of the second sentence.

Finally, if $\psi_1$ is separable and $d\geqslant p$, then $\sqrt[3]{y}\in\mathcal{F}$ by Remark~\ref{LemNew}. The definition of $\mathcal{F}$ implies that $\sqrt[3]{y}$ is separable over $\mathcal{Q}$, and this is a contradiction.
\end{proof}

\begin{lem}\label{LemNoIntExt}
Let $\theta$ be a root of $\psi$ and $K$ a field. If $\mathcal{F}\varsubsetneq  K\varsubsetneq \mathcal{F}(\theta)$, then $K=\mathcal{F}(\theta^3)$.
\end{lem}

\begin{proof}
Since $\psi$ is irreducible, we get $[\mathcal{F}(\theta):\mathcal{F}]=3p$. If $[K:\mathcal{F}]=p$, then $K$ is the field of all separable elements of $\mathcal{F}(\theta)$. In this case, $K=\mathcal{F}(\theta^3)$.

If $[K:\mathcal{F}]=3$, then $[\mathcal{F}(\theta) : K]=p$. In this case, $\theta$ is a root of an irreducible polynomial $\chi\in K[t]$ of degree $p$, and $\chi$ is separable. By Remark~\ref{LemNew}, we have $\chi=\sqrt[3]{\psi}$ which implies $\sqrt[3]{y},\sqrt[3]{z}\in K$. Note that the extension $\mathcal{Q}(\sqrt[3]{y},\sqrt[3]{z})\supset \mathcal{Q}$ has inseparability degree nine, which means that $[K:\mathcal{F}]_i\geq9$ because $\mathcal{F}\supset \mathcal{Q}$ is a separable extension. Therefore, $[K:\mathcal{F}]\neq 3$.
\end{proof}

\section{The matrices}
Let us recall a standard result we need in the rest of our paper. We denote by $\mathcal{C}(A)$ the \textit{centralizer} of a matrix $A\in\Mat_n(\mathbb{F})$, that is, the set of all matrices that commute with $A$. Clearly, $\mathcal{C}(A)$ is an $\mathbb{F}$-linear subspace of $\Mat_n(\mathbb{F})$, so we can speak of the dimension of $\mathcal{C}(A)$. A standard result of matrix theory states that there exists a nonsingular matrix $Q\in\Mat_n(\mathbb{F})$ such that $Q^{-1}AQ$ has \textit{rational normal form}. In other words, we have $Q^{-1}AQ=\operatorname{diag}[L(f_1),L(f_2),...,L(f_k)]$, where $L(f)$ denotes the companion matrix of a polynomial $f$, and the polynomials $f_i$ satisfy $f_{i+1}|f_i$ for all $i$. These polynomials $f_i$ are called the \textit{invariant factors} of $A$.

\begin{thr}\label{LemCla6}\cite[VIII.2, Theorem 2]{Gantmacher}
The centralizer of a matrix $A\in\Mat_{n}(\mathbb{F})$ has dimension $n_1+3n_{2}+\ldots+(2k-1) n_{k}$, where $n_1\geqslant\ldots\geqslant n_k$ are the degrees of the invariant factors of $A$.
\end{thr}

\begin{cor}\cite[VIII.2, Corollary 1]{Gantmacher}\label{ThrHJ}
If the characteristic polynomial of $A\in\Mat_n(\mathbb{F})$ is irreducible, then $\mathcal{C}(A)=\mathbb{F}[A]$.
\end{cor}



\begin{claim}\label{ClaDisF2}
Let $\mathcal{G}=\operatorname{diag}(G,G,G)$ be the block-diagonal matrix, where $G$ denotes the companion matrix of $\varphi$. If $X\in\Mat_{3p}(\mathcal{F})$ is the matrix whose $(i,j)$th entry equals $x_{ij}$, then  the distance between $\mathcal{G}$ and $X^{-1}\mathcal{G}X$ in $\Gamma(\mathcal{F},3p)$ is at least four.
\end{claim}

Before we start proving Claim~\ref{ClaDisF2}, we deduce the main result from it.

\begin{thr}\label{ThrDistSix}
Let $\Psi$ be an $\mathcal{F}$-linear operator on $\mathcal{F}^{3p}$ with characteristic polynomial $\psi$, and denote by $\Phi\in\Mat_{3p}(\mathcal{F})$ the matrix of $\Psi$ written with respect to a basis such that $\Psi^3$ has rational normal form. Then the distance between $\Phi$ and $X^{-1}\Phi X$ in $\Gamma(\mathcal{F},3p)$ is at least six.
\end{thr}

\begin{proof}
Suppose
$\Phi\leftrightarrow M_1\leftrightarrow M_2\leftrightarrow M_3\leftrightarrow M_4\leftrightarrow X^{-1}\Phi X$ is a path in $\Gamma(\mathcal{F},3p)$.
Corollary~\ref{ThrHJ} implies $\mathcal{C}(\Phi)=\mathcal{F}[\Phi]$; since $\psi$ is irreducible, $\mathcal{F}[\Phi]$ is isomorphic to $\mathcal{F}(\theta)$ as a field. By Lemma~\ref{LemNoIntExt}, the only matrices in $\mathcal{F}[\Phi]$ whose centralizer is larger are those in $\mathcal{F}[\Phi^3]$; moreover, all non-scalar matrices in $\mathcal{F}[\Phi^3]$ are polynomials in each other which means that these matrices have the same centralizer. Therefore, we can assume without loss of generality that $M_1=\Phi^3$ and $M_4=X^{-1}\Phi^3 X$. The minimal polynomial of $\Phi^3$ is $\varphi$, so $M_1$ is the block-diagonal matrix with three blocks equal to the companion matrix of $\varphi$. The rest follows from Claim~\ref{ClaDisF2}.
\end{proof}

\section{Proof of Claim~\ref{ClaDisF2}}

This is a final section which is intended to prove Claim~\ref{ClaDisF2}.
Let $A$ be a matrix over some extension of a field $\mathbb{F}$; we denote by $\operatorname{trdeg}(A,\mathbb{F})$ the transcendence degree of the field extension obtained from $\mathbb{F}$ by adjoining the entries of $A$. In the special case when $A$ is a matrix over $\mathcal{F}$, we write simply $\operatorname{trdeg}(A)$ instead of $\operatorname{trdeg}(A,\mathbb{F}_3(y,z))$. In particular, we have $\operatorname{trdeg}(X)=9p^2$ and $\operatorname{trdeg}(\mathcal{G})=0$ for the matrices defined in Claim~\ref{ClaDisF2}.

\begin{lem}\label{LemCla1}
Let $\mathbb{F}'$ be an extension of $\mathbb{F}$. If matrices $A,B,C$ over $\mathbb{F}'$ satisfy $B=C^{-1}AC$, then $\operatorname{trdeg}(C,\mathbb{F})\leqslant\operatorname{trdeg}(A,\mathbb{F})+\operatorname{trdeg}(B,\mathbb{F})+\dim\mathcal{C}(A)$.
\end{lem}

\begin{proof}
Let us denote by $K$ the algebraic closure of the field generated by entries of $A$ and $B$. Applying the Jordan normal form theorem, we see that there is a $J\in\Mat_{n}(K)$ such that $J^{-1}BJ=A$. Since $CJ\in \mathcal{C}(A)$, the entries of $CJ$ can be written as $K$-linear functions of $c=\dim\mathcal{C}(A)$ elements of $\mathbb{F}'$. Thus, the entries of $C$ are algebraic in these $c$ elements and the entries of $J$. The entries of $J$ are in turn algebraic in the entries of $A$ and $B$, so we have $\operatorname{trdeg}(C,\mathbb{F})\leqslant\operatorname{trdeg}(A,\mathbb{F})+\operatorname{trdeg}(B,\mathbb{F})+c$.
\end{proof}

\begin{cor}\label{CorCla1}
Assume $\mathbb{F}\subset \mathbb{F}'$ and matrices $A,C,D$ over $\mathbb{F}'$ satisfy $DC^{-1}AC=C^{-1}ACD$. Then $\operatorname{trdeg}(C,\mathbb{F})\leqslant\operatorname{trdeg}(A,\mathbb{F})+\operatorname{trdeg}(D,\mathbb{F})+\dim\mathcal{C}(A)+\dim\mathcal{C}(D)$.
\end{cor}

\begin{proof}
Let $K$ be the field obtained from $\mathbb{F}$ by adjoining the entries of $D$. Since the matrix $B=C^{-1}AC$ belongs to $\mathcal{C}(D)$, the entries of $B$ can be written as $K$-linear functions of $c=\dim\mathcal{C}(D)$ elements of $\mathbb{F}'$. We get $\operatorname{trdeg}(B,\mathbb{F})\leqslant \operatorname{trdeg}(D,\mathbb{F})+\dim\mathcal{C}(D)$, and now it suffices to substitute this into the right-hand side of the inequality of Lemma~\ref{LemCla1}.
\end{proof}

\begin{lem}\label{ObsCla1}
Every matrix $A\in\Mat_{n}(\mathcal{F})$ of rank $r$ satisfies $\operatorname{trdeg}(A)\leqslant 2nr-r^2$.
\end{lem}

\begin{proof}
We have $A=BC$, where $B\in\mathcal{F}^{n\times r}$ and $C\in\mathcal{F}^{r\times n}$. Since the product $BC$ is preserved by the transformation $(B,C)\to(BD,D^{-1}C)$, we can assume that $C$ has a unit $r\times r$ submatrix.
\end{proof}

\begin{lem}\label{LemCla5}
Assume that a matrix $N$ over $\mathcal{F}$ has one of the forms $$N_1=\begin{pmatrix}F_1 & 0 &0\\ 0 & F_2 &F_3\\0 & F_4 &F_5\end{pmatrix},\,\,
N_2=\begin{pmatrix}F_6 & F_8 &F_0\\ 0 & F_7&F_9\\0&0&F_6\end{pmatrix},$$ where $F_j\in\Mat_{p}(\mathcal{F})$. If $\operatorname{rank}(N)=p$, then $\operatorname{trdeg}(N)\leqslant 3.75p^2$.
\end{lem}

\begin{proof}
1. If $F_1$ has rank $k$, then the bottom right $2p\times 2p$ submatrix of $N_1$ has rank $p-k$; by Lemma~\ref{ObsCla1}, we get $\operatorname{trdeg}(N_1)\leqslant 2pk-k^2+4p(p-k)-(p-k)^2=3p^2-2k^2\leqslant 3p^2$.

2. We have $\operatorname{rank}\,F_6\leqslant p/2$, and Lemma~\ref{ObsCla1} shows that $\operatorname{trdeg}(F_6)\leqslant 0.75p^2$. Finally, the upper right $2p\times 2p$ submatrix of $N_2$ has rank at most $p$, so its transcendence degree does not exceed $3p^2$ again by Lemma~\ref{ObsCla1}.
\end{proof}

%


\begin{lem}\label{LemGant}
Let $\mathbb{F}$ be a field over which polynomials of degrees $1,\ldots,d-1$ have roots. If a matrix $M\in\Mat_n(\mathbb{F})$ has no eigenvalue in $\mathbb{F}$, then $\dim\mathcal{C}(M)\leqslant n^2/d$.
\end{lem}

\begin{proof}
Let $n_1\geqslant\ldots\geqslant n_k$ be the degrees of the invariant factors. Since $M$ has no eigenvalue in $\mathbb{F}$, we have $n_i\geqslant d$ for all $i\in\{1,\ldots,k\}$. Denoting $\delta_i=n_i/d$, we get $\delta_i\geqslant1$ and also $\delta_1+\ldots+\delta_k=n/d$. Now we have $2i+1\leqslant 2\delta_1+\ldots+2\delta_{i}+1$ for all $i$, and Theorem~\ref{LemCla6} implies $\dim\mathcal{C}(M)/d=\delta_1+3\delta_2+\ldots+(2k-1)\delta_k$, which is $\leqslant \delta_1+(2\delta_1+1)\delta_2+(2\delta_1+2\delta_2+1)\delta_3+\ldots+(2\delta_1+\ldots+2\delta_{k-1}+1)\delta_k=\delta_1+\ldots+\delta_k+2\sum_{i<j}\delta_i\delta_j\leqslant (\delta_1+\ldots+\delta_k)^2=n^2/d^2$. 
\end{proof}

\begin{lem}\label{LemCla2}
Let $N\in\mathcal{C}(\mathcal{G})$ be a non-scalar matrix, where $\mathcal{G}$ is the matrix from Claim~\ref{ClaDisF2}. If $\dim\mathcal{C}(N)>3p^2$, then

\noindent (1) there is a non-scalar matrix $M\in\mathcal{C}(\mathcal{G})$ such that $\mathcal{C}(N)\subset\mathcal{C}(M)$ and either $M^2=M$ or $M^2=0$;

\noindent (2) there is a matrix $U\in\mathcal{C}(\mathcal{G})$ such that $\mathcal{C}(UNU^{-1})$ is contained in the centralizer of one of the matrices
$$M_1=\begin{pmatrix}I & 0 &0\\ 0 & 0&0\\0&0&0\end{pmatrix},\,\,
M_2=\begin{pmatrix}0 & 0 & I\\ 0 & 0&0\\0&0&0\end{pmatrix},$$
where every block has size $p\times p$.
\end{lem}

\begin{proof}
One can note that the elements of $\mathcal{C}(\mathcal{G})$ are the 
$3p \times 3p$ matrices such that, when they are written as $3\times3$ arrays of $p\times p$ blocks, all of those $p\times p$ blocks commute with $G$. By Corollary~\ref{ThrHJ}, these blocks are actually elements of $\mathcal{F}[G]$; in other words, we can think of $\mathcal{C}(\mathcal{G})$ as the set of $3\times 3$ matrices over the field $\mathcal{F}[G]$.

If the assertion (1) of the lemma is true, we can compute the Jordan normal form of $M$ as a $3\times3$ matrix over $\mathcal{G}$. If $M^2=0$, then the Jordan normal form of $M$ is the $M_2$ of statement (2); if $M^2=M$, then $\mathcal{C}(M)=\mathcal{C}(I-M)$ and either $M$ or $I-M$ has Jordan normal form equal to $M_1$. Now it suffices to prove (1). 

If $N$ has no eigenvalue in $\mathcal{F}$, then Lemma~\ref{LemGant} with $d=3$ implies $\dim\mathcal{C}(N)\leqslant3p^2$, which is a contradiction. (Lemma~\ref{LemGant} is applicable because every degree two polynomial over $\mathcal{F}$ is reducible by Theorem~\ref{CorIntExt}.) We see that $N$ has an eigenvalue $\lambda\in\mathcal{F}$, and then $L=N-\lambda I$ is non-invertible and $\mathcal{C}(N)=\mathcal{C}(L)$. If $L$ is nilpotent, then either $L$ or $L^2$ is a nonzero square-zero matrix, and we have $\mathcal{C}(L)\subset\mathcal{C}(L^2)$.

If $L$ is not nilpotent, then we construct the rational normal form of $L^2$. That is, we find a matrix $U\in\mathcal{C}(\mathcal{G})$ such that $UL^2U^{-1}=\left(\begin{smallmatrix}0&0\\0&L'\end{smallmatrix}\right)$, where $L'$ is either a $p\times p$ or $2p\times 2p$ invertible matrix. The matrix $L_0$ obtained by replacing $L'$ with the unit matrix of the relevant size is a polynomial in $UL^2U^{-1}$, so we get $L_0=L_0^2$ and $\mathcal{C}(ULU^{-1})\subset\mathcal{C}(UL^2U^{-1})\subset\mathcal{C}(L_0)$.
\end{proof}

\begin{lem}\label{LemCHT}
For any $U\in\mathcal{C}(\mathcal{G})$, we have $\operatorname{trdeg}(U)\leqslant 9p$.
\end{lem}

\begin{proof}
By the Cayley--Hamilton theorem, the matrix $G^p$ (where $G$ is as in Claim 3.3) is an $\mathcal{F}$-linear combination of matrices $I,G,\ldots,G^{p-1}$. In other words, any matrix $P\in\mathcal{F}[G]$ can be written as $\lambda_0 I+\lambda_1 G+\ldots+\lambda_{p-1} G^{p-1}$ with $\lambda_i\in\mathcal{F}$. Since $G$ is a companion matrix of a polynomial with coefficients in $\mathbb{F}_3[x,y]$, we get $\operatorname{trdeg}(G)=0$ and $\operatorname{trdeg}(P)\leqslant p$. As explained in the proof of Lemma~\ref{LemCla2}, matrices in $\mathcal{C}(\mathcal{G})$ are $3p \times 3p$ matrices such that, when they are written as $3 \times 3$ arrays of $p \times p$ blocks, all of those nine blocks belong to $\mathcal{F}[G]$.
\end{proof}

We are now ready to complete the proof of Claim~\ref{ClaDisF2}.

\begin{proof}[Proof of Claim~\ref{ClaDisF2}]
Suppose
$\mathcal{G}\leftrightarrow S\leftrightarrow X^{-1}TX\leftrightarrow X^{-1}\mathcal{G} X$ is a path in $\Gamma(\mathcal{F},3p)$. By Corollary~\ref{CorCla1} with $X$, $T$, $S$ in the roles of $C$, $A$, $D$, respectively, we have $\operatorname{trdeg}(X)\leqslant\operatorname{trdeg}(T)+\operatorname{trdeg}(S)+\dim\mathcal{C}(T)+\dim\mathcal{C}(S)$. Lemma~\ref{LemCHT} implies $\operatorname{trdeg}(T)+\operatorname{trdeg}(S)\leqslant 18p$, so that $\dim\mathcal{C}(T)+\dim\mathcal{C}(S)\geqslant 9p^2-18p>8p^2$ (since $p>18$). 

Thus, at least one of $\dim\mathcal{C}(S)$, $\dim\mathcal{C}(T)$ must be $>3p^2$. Suppose $\dim\mathcal{C}(S)> 3p^2$. Note that matrices commuting with $M_1$, $M_2$ as in Lemma~\ref{LemCla2} have the forms $N_1$, $N_2$ as in Lemma~\ref{LemCla5}, respectively. Therefore, Lemma~\ref{LemCla2} implies $\dim\mathcal{C}(S)\leqslant 5p^2$, which implies $\dim\mathcal{C}(T)>3p^2$ by applying the conclusion of the above paragraph. If, instead of $\dim\mathcal{C}(S)> 3p^2$, we assume $\dim\mathcal{C}(T)>3p^2$, then we similarly get $\dim\mathcal{C}(S)> 3p^2$. So both these bounds hold.

Since we can replace $S$ and $T$ by any non-scalar matrices with the same or larger centralizers, Lemma~\ref{LemCla2} allows us to assume that $S=UM_iU^{-1}$ and $T=V^{-1}M_jV$, for some $i,j\in\{1,2\}$ and $U,V\in\mathcal{C}(\mathcal{G})$. Since any matrix commuting with $M_i$ has the form $N_i$ as in Lemma~\ref{LemCla5}, we get $X^{-1}TX=UN_iU^{-1}$, which implies $N_i=(VXU)^{-1}M_j(VXU)$; we have
\begin{equation}\label{eqCla}\operatorname{trdeg}(VXU)\leqslant\operatorname{trdeg}(N_i)+\operatorname{trdeg}(M_j)+\dim\mathcal{C}(M_j)\end{equation}
by Lemma~\ref{LemCla1}. Since $\operatorname{rank} N_i=\operatorname{rank} T=\operatorname{rank} M_j=p$, Lemma~\ref{LemCla5} implies $\operatorname{trdeg}(N_i)\leqslant 3.75p^2$; we also get $\operatorname{trdeg}(M_j)=0$, $\dim\mathcal{C}(M_j)=5p^2$ straightforwardly. Finally, we note that $\operatorname{trdeg}(X)\leqslant\operatorname{trdeg}(VXU)+\operatorname{trdeg}(V^{-1})+\operatorname{trdeg}(U^{-1})$, and Lemma~\ref{LemCHT} implies $\operatorname{trdeg}(V^{-1})+\operatorname{trdeg}(U^{-1})\leqslant 18p$. We get $\operatorname{trdeg}(VXU)\geqslant 9p^2-18p$, so that~\eqref{eqCla} implies $9p^2-18p\leqslant 3.75p^2+0+5p^2$, i.e., $0.25p^2\leqslant18p$. Since $p\geqslant73$, this is a contradiction.
\end{proof}

The proof of the main result is now complete. Theorems~\ref{CorIntExt} and~\ref{ThrDistSix} show that $\Gamma(\mathcal{F},3p)$ is a connected graph with diameter greater than five, which disproves Conjecture~\ref{ConjConn5}.
The distances in commuting graphs of matrix algebras cannot exceed six as Theorem~17 of~\cite{AMRP} shows, so the diameter of $\Gamma(\mathcal{F},3p)$ equals six.

%

\bigskip

The idea of our construction comes from the paper~\cite{Lipman}, which contains the proof of a statement similar to Lemma~\ref{LemNoIntExt}. I would like to thank the anonymous reviewers for their interest to this project, careful reading of the preliminary versions, and numerous helpful suggestions. I am grateful to Alexander Guterman for interesting discussions on this topic.

\end{document}